\documentclass{scrartcl}

\KOMAoptions{headings=normal}
\usepackage[usenames, dvipsnames]{xcolor}

\usepackage{tikz}
\usepackage{tikz-cd}
\usepackage{braket}
\usepackage{graphicx}
\usepackage{amsmath}
\usepackage{amsthm}
\usepackage{csquotes}

\usepackage[T1]{fontenc}
\usepackage[bitstream-charter]{mathdesign}
\usepackage[colorlinks=true, linkcolor=NavyBlue, citecolor = OliveGreen]{hyperref}

\theoremstyle{definition}
\newtheorem{sub}{}[section]
\newtheorem{defn}[sub]{Definition}

\newtheorem{Remark}[sub]{Remark}
\newtheorem{Example}[sub]{Example}

\newtheorem{constr}{Construction}
\newtheorem{notation}[sub]{Notation}

\theoremstyle{plain}
\newtheorem{thm}[sub]{Theorem}
\newtheorem{lem}[sub]{Lemma}

\newtheorem{coro}[sub]{Corollary}
\newtheorem{prop}[sub]{Proposition}

\newcommand{\mcyl}{(\Delta^1)^\sharp}
\newcommand{\1}{\{1\}}
\newcommand{\ian}{$I$-anodyne }
\newcommand{\psh}{\mathrm{PSh}}
\newcommand{\ifib}{$I$-fibrant }
\newcommand{\Hom}{\mathrm{Hom}}

\title{Covariant \& Contravariant Homotopy Theories}
\author{Hoang Kim Nguyen}
\date{}

\begin{document}

\begin{titlepage}
\maketitle
\begin{abstract}
Given a locally presentable category together with a suitable functorial cylinder object, we construct model structures which are sensitive to the `direction' of the cylinder. We show that the Covariant and Contravariant model structures on simplicial sets as well as the coCartesian and Cartesian model structures on marked simplicial sets are examples of our formalism. In this setting, notions of final and initial maps and smooth and proper maps arise very naturally and we will identify these maps in the examples.
\end{abstract}

\tableofcontents

\end{titlepage}

\section{Introduction}	

\begin{sub}
Suppose we have a topos $E$ together with a suitable functorial cylinder object $I\colon E \to E$. In the articles \cite{cisinskitop} and \cite{cisinski}, Cisinski shows that one can generate from this datum a model structure on $E$ in which the cofibrations are precisely the monomorphisms and the fibrant objects are characterized by a right lifting property against a set of morphisms constructed out of the cylinder object, the $I$-anodyne extensions. Prominent examples of this construction are the Kan-Quillen model structure and the Joyal model structure on simplicial sets. Cisinski's theory has been generalized by Olschok to the setting of locally presentable categories \cite{olschok} with a suitable weak factorization system playing the role of the weak factorization system given by monomorphisms and trivial fibrations in a topos.
\end{sub}

\begin{sub}
The goal of this article is to modify Cisinski's theory so that the `direction' of the cylinder (and thus the direction of homotopy) matters. This will naturally lead to two in general \emph{distinct} model structures arising from a functorial cylinder object, see Theorem \ref{righthm} and Theorem \ref{lefthm} and we suggestively call them Covariant and Contravariant model structures, depending on the direction that we chose. The terminology is motivated by the two main examples that we consider in this article. Given a simplicial set $A$, we show that the two model structures on the slice category $\mathbf{sSet}/A$ that we construct here, Theorem \ref{covariant}, correspond to Joyal's Covariant and Contravariant model structures, see \cite{joyalquadern} and also \cite{lurie} and \cite{cisinskibook} for different proofs of the existence of these model structures. Similarly, we show that the two model structures on the slice category $\mathbf{sSet}^+/A^\sharp$, Theorem \ref{markedcov}, correspond to Lurie's coCartesian and Cartesian model structures \cite{lurie}. In particular, both examples arise as instances of the same general formalism. Thus, the homotopy theories we construct here are a natural setting in which we can speak about homotopy theories behaving in a `covariant' or `contravariant' fashion.
\end{sub}

\begin{sub}
In this setting, the notions of \emph{final} and \emph{initial} maps arise very naturally. Given a suitable cylinder on a locally presentable category $C$ and an object $A\in C$, we can consider Covariant and Contravariant model structures on the slice categories $C/A$. A final (resp. initial) map is defined as a map in $C$, which is an equivalence in the Contravariant (resp. Covariant) model structure on $C/A$ for \emph{all} objects $A$. We will see that the final and initial maps are completely determined by the cylinder, see Proposition \ref{leftcancel} for a precise statement. This will also lead to the notion of \emph{smooth} and \emph{proper} maps and we will see, in a quite elementary way, that left (resp. right) fibrations and coCartesian (resp. Cartesian) fibrations are proper (resp. smooth) with respect to the Contravariant model structures on simplicial sets, see Theorem \ref{properleft} and Corollary \ref{markedproperleft}.
\end{sub}

\begin{sub}
Finally, we want to mention that the proofs of the next two sections in this article are mostly due to Cisinski, although in a less general setting. Nevertheless we gave full proofs, just to verify that his arguments indeed carry over. Our main source of inspiration is \cite[Section 2.4]{cisinskibook}.
\end{sub}

\medskip

\emph{Acknowledgments:} This article is part of the author's PhD Thesis written under the supervision of George Raptis and Ulrich Bunke at the University of Regensburg. The author thanks both of them for their guidance and support. Furthermore, the author thanks Denis-Charles Cisinski for many helpful discussions on this topic. Finally, the author gratefully acknowledges support from the Deutsche Forschungsgesellschaft through the \emph{S{}FB 1085 -- Higher Invariants}.

\section{Covariant \& Contravariant model structures}\label{abstrcontr}

\begin{sub}
We fix once and for all a locally presentable category $C$ together with a cofibrantly generated weak factorization system $(\mathcal L, \mathcal R)$. Moreover, we assume that for each object $X\in C$ the canonical morphism $\emptyset \to X$ is in the left class $\mathcal L$.
\end{sub}

\begin{defn}
Let $X\in C$ be an object. A \emph{cylinder on X} is a commutative diagram
\[
\begin{tikzcd}
X\drar[swap]{\partial_0} \arrow[bend left]{drr}{id_X} & & \\
 & IX \rar{\sigma} & X\\
X \urar{\partial_1} \arrow[bend right, swap]{urr}{id_X} & &
\end{tikzcd}
\]
where the induced map $\partial_0 \sqcup \partial_1 \colon X\sqcup X \to IX$ is in the left class $\mathcal L$.
\end{defn}     

\begin{sub}
Consider the endomorphism category $End(C)$. This is a monoidal category with monoidal product given by composition. It acts on the left on $C$ by
\begin{align}
End(C)\times C &\to C\\
(F,X)&\mapsto F\otimes X = F(X).
\end{align}
In particular, for any natural transformation $\eta\colon F\Rightarrow G$ and any morphism $f\colon X\to Y \in C$ we obtain a morphism
\[
\eta \otimes f \colon F\otimes X \to G\otimes Y.
\]
\end{sub}

\begin{defn}
A \emph{functorial cylinder object} on the category $C$ is an endofunctor $I\colon C \to C$ together with natural transformations
\begin{itemize}
	\item $\partial_0 \sqcup \partial_1\colon id_{C}\sqcup id_{C} \Rightarrow I$
	\item $\sigma \colon I \Rightarrow id_{C}$
\end{itemize}
such that for each $X \in C$, evaluation at $X$ defines a cylinder on $X$. 
\end{defn}

\begin{notation}\label{boxproduct}
Suppose we have a functorial cylinder $(I,\partial_0,\partial_1,\sigma)$ on $C$. We denote $\partial I := id_C \sqcup id_C$. We thus have natural transformations
\begin{itemize}
 	\item $\partial_0 \sqcup \partial_1 \colon \partial I \Rightarrow I$
 	\item $\partial_i \otimes id \colon \{i\}\otimes id \cong id \Rightarrow I$ for $i=0,1$.
 \end{itemize} 
The cylinder induces three operations on the morphisms of $C$. Given a morphism $i\colon K \to L \in C$ we obtain a commutative square
\[
\begin{tikzcd}
\partial I \otimes K \rar \dar & I\otimes K\dar\\ 
\partial I \otimes L \rar & I\otimes L.
\end{tikzcd}
\]
We denote the induced map from the pushout
\[
\partial I \boxtimes i \colon \partial I \otimes L \sqcup_{\partial I \otimes K} I \otimes K \to I \otimes L.
\]
Similarly for $i=0,1$ we have a commutative square
\[ 
\begin{tikzcd}
\{i\}\otimes K \rar \dar & I\otimes K\dar\\ 
\{i\}\otimes L \rar & I\otimes L
\end{tikzcd}
\]
and we denote the induced map from the pushout
\[
\partial_i \boxtimes i \colon \{i\} \otimes L \sqcup_{\{i\}\otimes K} I \otimes K \to I \otimes L.
\]
\end{notation}

\begin{sub}
Given a functorial cylinder, we impose additional compatibility conditions with respect to the weak factorization system $(\mathcal L, \mathcal R)$. 
\end{sub}

\begin{defn}
A functorial cylinder is called \emph{exact} with respect to $(\mathcal L, \mathcal R)$ if the following hold:
\begin{itemize}
	\item The functor $I$ commutes with small colimits.
	\item For any morphism $j\colon K \to L \in \mathcal L$ the morphism $\partial I \boxtimes j$ is in $\mathcal L$.
	\item For any morphism $j\colon K\to L \in \mathcal L$ the morphism $\partial_j \boxtimes j$ is in $\mathcal L$ for $i=0,1$.
\end{itemize}
\end{defn}

\begin{Example}\label{separatingsegment}
Let $A$ be a small category and consider its category of presheaves $\psh(A)$. This admits a weak factorization system where the left class is given by the class of monomorphisms. We will call the right class the class of \emph{trivial fibrations} and denote the weak factorization system by (Mono, Triv). Let $I$ be a presheaf together with two maps from the terminal presheaf $\partial_i \colon \ast \to I$, where $i=0,1$, such that
\[
\begin{tikzcd}
\emptyset \rar \dar & \ast \dar{\partial_0}\\
\ast \rar[swap]{\partial_1} & I
\end{tikzcd}
\]
is cartesian. Then the endofunctor
\[
I \times (\cdot)\colon \psh(A) \to \psh(A)
\]
defines an exact cylinder with respect to the weak factorization system (Mono, Triv). Indeed, for a presheaf $X$ the structure maps are given by $\partial_i \times id_X\colon X\to I\times X$ and $\sigma \colon I\times X \to X$ is given by the projection to $X$. For any monomorphism of presheaves $j\colon K \to L$, we have a cartesian square
\[
\begin{tikzcd}
\partial I \times K \cong K \sqcup K \rar \dar & I\times K\dar \\
\partial I \times L \cong L \sqcup L \rar  & I\times L
\end{tikzcd}
\]
since colimits are universal. It follows that the map $\partial I \boxtimes j$ is a monomorphism and a similar argument shows that $\partial_i \boxtimes j$ is a monomorphism for $j=0,1$. Since the category of presheaves is cartesian closed, the functor $I\times (\cdot)$ commutes with colimits hence is exact with respect to (Mono, Triv).
\end{Example}

\begin{defn}
A class of morphisms $An^r(I)\subseteq \mathcal L$ is called a class of \emph{right I-anodyne extensions} if the following axioms are satisfied.
\begin{itemize}
	\item There exists a (small) set of morphisms $\Lambda\subseteq \mathcal L$ such that we have $An^r(I)= l(r(\Lambda))$.
	\item For any $i\colon K\to L\in \mathcal L$, the induced map $\partial_1\boxtimes i$ is in $An^r(I)$.
	\item For any $i\colon K\to L\in An^r(I)$, the map $\partial I \boxtimes i$ is also in $An^r(I)$
\end{itemize}
A \emph{right homotopical structure} on $C$ is the datum of an exact cylinder $(I,\partial_0,\partial_1, \sigma)$ together with a choice of right \ian extensions $An^r(I)$. A \emph{right I-fibration} is a morphism of $C$ having the right lifting property with respect to the class of right $I$-anodyne extensions. An object is \emph{right I-fibrant} if its canonical map to the terminal object is a right $I$-fibration.
\end{defn}

\begin{sub}
Dually, we may define the following.
\end{sub}

\begin{defn}
A class of morphisms $An^l(I)\subseteq \mathcal L$ is called a class of \emph{left I-anodyne extensions} if the following axioms are satisfied.
\begin{itemize}
	\item There exists a (small) set of morphisms $\Lambda\subseteq \mathcal L$ such that we have $An^r(I)= l(r(\Lambda))$.
	\item For any $i\colon K\to L\in \mathcal L$, the induced map $\partial_0\boxtimes i$ is in $An^r(I)$.
	\item For any $i\colon K\to L\in An^r(I)$, the map $\partial I \boxtimes i$ is also in $An^r(I)$
\end{itemize}
A \emph{left homotopical structure} on $C$ is the datum of an exact cylinder $(I,\partial_0,\partial_1, \sigma)$ together with a choice of left $I$-anodyne extensions $An^l(I)$. A \emph{left I-fibration} is a morphism of $C$ having the right lifting property with respect to the class of left $I$-anodyne extensions. An object is \emph{left I-fibrant} if its canonical map to the terminal object is a left $I$-fibration.
\end{defn}

\begin{Remark}\label{plainanodyne}
Our definition of \emph{right} (and \emph{left}) $I$-anodyne extension differs from Cisinski's notion of (plain) \emph{$I$-anodyne extensions} in the following way. In Cisinski's axioms it is required that for any morphism $i\colon K\to L\in \mathcal L$ \emph{both} morphisms
\begin{itemize}
	\item $\partial_0 \boxtimes i$ and
	\item $\partial_1 \boxtimes i$
\end{itemize}
are $I$-anodyne extensions, while we only require the second one for our notion of right $I$-anodyne extensions. This gives a \emph{direction} for right $I$-anodyne extensions. For example, for any object $K \in  C$ the morphism $\1 \otimes K \to I\otimes K$ is right $I$-anodyne while the morphism $\{0\} \otimes K \to I\otimes K$ is not.
\end{Remark}

\begin{sub}
We will see that a class of right (or left) $I$-anodyne extensions always exists. For example we may take the class $\mathcal L$ to be a class of right $I$-anodyne extensions. At the end of this section, we will consider right $I$-anodyne extensions arising from an \emph{elementary homotopical datum}. But first, our main goal of this section is to prove that any right and any left homotopical structure gives rise to a model structure on $C$.
\end{sub}

\begin{defn}
Let $f, g \colon X\to Y$ be two morphisms. An \emph{I-homotopy} from $f$ to $g$ is a morphism
\[
h\colon I \otimes X \to Y
\]
such that $h(\partial_0 \otimes id_X) = f$ and $h(\partial_1 \otimes id_X)= g$. We denote by $[X,Y]_I$ the quotient of $\hom_{C}(X,Y)$ by the equivalence relation generated by the notion of $I$-homotopy. We denote by $\mathrm{Ho}_I( C)$ the category having the same objects as $C$ and morphism sets given by the quotients $[X,Y]_I$. We will refer to this category as the \emph{I-homotopy category of $C$}. We have a canonical projection $C \to \mathrm{Ho}_I(C)$. A morphism is an \emph{I-homotopy equivalence} if its image in the $I$-homotopy category is an isomorphism.
\end{defn}

\begin{Remark}
The functoriality of the cylinder ensures that $\mathrm{Ho}_I(C)$ is indeed a category.
\end{Remark}

\begin{sub}
We will prove the following pair of Theorems.
\end{sub}

\begin{thm}\label{righthm}
Suppose we have a right homotopical structure on $C$. Then there exists a unique model structure on $C$ with the following description.
\begin{enumerate}
	\item The class of cofibrations is precisely the class $\mathcal L$.
	\item A morphism $f\colon A \to B$ is a weak equivalence if and only if for all right $I$-fibrant objects $W\in C$, the induced morphism
	\[
	f^\ast \colon [B,W]_I\to [A,W]_I
	\]
	is bijective.
\end{enumerate}
Furthermore, an object is fibrant if and only if it is right $I$-fibrant and a morphism between right $I$-fibrant objects is a fibration if and only if it is a right $I$-fibration.
\end{thm}

\begin{thm}\label{lefthm}
Suppose we have a left homotopical structure on $C$. Then there exists a unique model structure on $C$ with the following description.
\begin{enumerate}
	\item The class of cofibrations is precisely the class $\mathcal L$.
	\item A morphism $f\colon A \to B$ is a weak equivalence if and only if for all left $I$-fibrant objects $W\in C$, the induced morphism
	\[
	f^\ast \colon [B,W]_I\to [A,W]_I
	\]
	is bijective.
\end{enumerate}
Furthermore, an object is fibrant if and only if it is left $I$-fibrant and a morphism between left $I$-fibrant objects is a fibration if and only if it is a left $I$-fibration.
\end{thm}

\begin{sub}
The proof requires several steps. We will only focus on right homotopical structures. The proof for left homotopical structures is entirely analogous, requiring only minor modifications in the direction of the homotopy. The basis is Jeff Smith's recognition theorem for combinatorial model categories. We will use the following variant due to Carlos Simpson.
\end{sub}

\begin{thm}\label{simpson}
Let $\mathbf{M}$ be a locally presentable category and $S$ and $\Lambda$ sets of morphisms such that $\Lambda \subset l(r(S))$. Define a morphism $f\colon A\to B$ to be a weak equivalence if and only if there exists a diagram
\[
\begin{tikzcd}
A\rar \dar[swap]{f} & X\dar \\
B \rar & Y
\end{tikzcd}
\]
such that the horizontal arrows are transfinite compositions of pushouts of morphisms in $\Lambda$ and the right vertical arrow is in $r(S)$. Define the class of cofibrations to be $l(r(S))$ and suppose furthermore that
\begin{enumerate}
	\item the domains of $\ I$ and $\Lambda$ are cofibrant,
	\item the class of weak equivalences above is closed under retracts and satisfies 2-out-of-3,
	\item the class of trivial cofibrations is closed under pushouts and transfinite compositions. 
\end{enumerate}
Then there exists a cofibrantly generated model structure on $\mathbf{M}$ with the given class of cofibrations and weak equivalences.
\end{thm}

\begin{proof}
This is \cite[Theorem 8.7.3]{simpson}
\end{proof}

\begin{sub}
In our situation, the set $\Lambda$ will be the generating set of right \ian extensions and the set $S$ will be a generating set for $\mathcal L$ (recall that $(\mathcal L, \mathcal R)$ was assumed to be cofibrantly generated). It is clear that the domains of $S$ and $\Lambda$ are cofibrant by our assumptions on $(\mathcal L, \mathcal R)$ and that our class of weak equivalences is closed under retracts and satisfies 2-out-of-3. Thus, our task will be to show that our class of weak equivalences satisfy the description of Simpson's theorem and that the trivial cofibrations are closed under pushouts and transfinite compositions. Along the way, our proofs will also imply the description of fibrations we gave in our theorem. We first show that any right \ian is a weak equivalence in the sense of Theorem \ref{righthm}.
\end{sub}

\begin{lem}
If $\ W$ is right \ifib, then $I$-homotopy is an equivalence relation on the set $\ \Hom(X,W)$ for any object $\ X$.
\end{lem}

\begin{proof}
Consider three morphisms
\[
u,v,w\colon X \to W.
\]	
Suppose we have homotopies
\[
h\colon I\otimes X \to W \quad \text{such that}\ h(\partial_0 \otimes 1_X)= u,\ h(\partial_1 \otimes 1_X)= w
\]
and
\[
k\colon I\otimes X \to W \quad \text{such that}\ h(\partial_0 \otimes 1_X)= v,\ h(\partial_1 \otimes 1_X)= w.
\]
We will show that there exists an $I$-homotopy from $u$ to $v$.

We have a map
\[
((h,k), \sigma \otimes w)\colon I\otimes \partial I \otimes X \sqcup_{\1 \otimes \partial I \otimes X} \{1\}\otimes I \otimes X\to W
\]
and the map
\[
I\otimes \partial I \otimes X \sqcup_{\1 \otimes \partial I \otimes X} \{1\}\otimes I \otimes X \to I\otimes I\otimes X
\]
is a right $I$-anodyne extension since $\partial I \otimes X \to I \otimes X \in \mathcal L$. By assumption $W$ is $I$-fibrant, thus we have a homotopy
\[
H\colon I\otimes I\otimes X\to W
\]
such that
\[
H(1_I \otimes \partial_0 \otimes 1_X)=h
\]
and
\[
H(1_I \otimes \partial_1 \otimes 1_X)=k.
\]
Moreover, we have
\[
H(\partial_1 \otimes 1_I\otimes 1_X)= \sigma \otimes w
\]
Now define an $I$-homotopy $\eta\colon I\otimes X \to W$ by the formula
\[
\eta = H(\partial_0 \otimes 1_I\otimes 1_X).
\]
We then have
\[
\eta (\partial_0 \otimes 1_X)= H(\partial_0 \otimes \partial_0\otimes 1_X)= h(\partial_0\otimes 1_X)= u
\]
and 
\[
\eta (\partial_1 \otimes 1_X)= H(\partial_0 \otimes \partial_1 \otimes 1_X )= k(\partial_0 \otimes 1_X)=v.
\]
Thus $\eta$ defines a homotopy from $u$ to $v$.

Now if $h$ is the constant homotopy at $u$ and $k$ is a homotopy from $v$ to $u$, then $\eta$ provides a homotopy from $u$ to $v$ showing that $I$-homotopy is symmetric. Transitivity follows from the above construction and symmetry.
\end{proof}

\begin{prop}\label{anodyneequivalence}
Any right I-anodyne extension is a weak equivalence.
\end{prop}

\begin{proof}
Let $f\colon A \to B$ be a right $I$-anodyne extension and let $W$ be right $I$-fibrant. It is enough to show that
\[
f^\ast \colon [B,W]_I \to [A,W]_I
\]
is injective. Thus let $\beta_0,\beta_1\colon B \to W$ be two morphisms such that $\beta_0 f$ is homotopic to $\beta_1 f$. By the above lemma, there exists a homotopy
\[
h\colon I\otimes A\to W
\]
such that $h_0= \beta_0 f$ and $h_1 = \beta_1 f$. This gives rise to a lifting problem
\[
\begin{tikzcd}[column sep = huge]
I\otimes A \sqcup_{\partial I \otimes A} \partial I \otimes B \rar{(h, \beta_0 \sqcup \beta_1)}\dar  & W\\
I\otimes B \urar[dashed]&
\end{tikzcd}
\]
Since $f$ is right $I$-anodyne, the vertical map is also right $I$-anodyne and hence, since $W$ is right $I$-fibrant, the lifting problem admits a solution. This provides a homotopy from $\beta_0$ to $\beta_1$.
\end{proof}

\begin{sub}
Now suppose we have a commutative square
\[
\begin{tikzcd}
A\dar[swap]{f}\rar & X\dar \\
B \rar & Y
\end{tikzcd}
\]
in which the horizontal maps are transfinite compositions of pushouts of $\Lambda$ and the map $X\to Y$ is in the class $\mathcal R$. In particular, the horizontal maps are right $I$-anodyne extensions and hence weak equivalences by the above proposition. To conclude that $f$ is a weak equivalence, we need to show morphisms in the class $\mathcal R$ are weak equivalences. We can actually show a stronger statement. To this end we introduce a particularly nice class of $I$-homotopy equivalences (and hence weak equivalences).
\end{sub}

\begin{defn}
A morphism $i \colon A\to X$ is called a \emph{right deformation retract} if there exists a morphism $r\colon X \to A$ and a homotopy $h\colon I\otimes X \to X$ such that 
\begin{enumerate}
	\item $ri= id_A$
	\item $h_0 = id_X$ and $h_1= ir$
	\item $h(id_I\otimes i)= \sigma \otimes i$.
\end{enumerate}
A morphism $r\colon X \to A$ is called a \emph{dual of a right deformation retract} if there exists a map $i\colon A \to X$ and a homotopy $h\colon I\otimes X \to X$ such that
\begin{enumerate}
	\item $ri=id_A$
	\item $h_0 = id_X$ and $h_1=ir$
	\item $rh=\sigma\otimes  r$.
\end{enumerate}
\end{defn}

\begin{prop}\label{trivdual}
Any map $f\colon X\to Y \in \mathcal R$ is the dual of a right deformation retract.
\end{prop} 

\begin{proof}
We find a section $s\colon Y \to X$ via the lifting problem
\[
\begin{tikzcd}
\emptyset \rar \dar & X\dar{r} \\
Y \rar{id} \urar[dashed]{s} & Y
\end{tikzcd}
\]
since $\emptyset \to Y \in \mathcal L$. We have a lifting problem
\[
\begin{tikzcd}[column sep = huge]
\partial I \otimes X \rar{(id_X \sqcup sr)} \dar & X\dar{r}\\
I\otimes X\urar[dashed] \rar{\sigma\otimes r}\rar[dashed] & Y
\end{tikzcd}
\]
which admits a lift since the left vertical map is in $\mathcal L$ by exactness of the cylinder, verifying that $r$ is the dual of a right deformation retract.
\end{proof}

\begin{sub}
In conclusion we have shown that whenever we have a commutative square
\[
\begin{tikzcd}
A\dar[swap]{f}\rar & X\dar \\
B \rar & Y
\end{tikzcd}
\]
in which the horizontal maps are right $I$-anodyne extensions and the map $X\to Y$ is in the class $\mathcal R$, then $f$ is a weak equivalence. In particular, any map satisfying Simpson's description is a weak equivalence in our sense. 

Conversely, suppose that $f\colon A \to B$ is a weak equivalence. By the small object argument we find a right $I$-anodyne extension $B\to Y$ such that $Y$ is right $I$-fibrant. Again by the small object argument, we factorize the composition $A\to B \to Y$ as a right $I$-anodyne extension followed by a right $I$-fibration to obtain the square
\[
\begin{tikzcd}
A\dar[swap]{f} \rar & X\dar\\
B\rar & Y.
\end{tikzcd}
\]
By the 2-out-of-3 property, the morphism $X\to Y$ is a weak equivalence. By construction, it is also a right $I$-fibration with right $I$-fibrant domain. Thus we need to show that right $I$-fibrations with right $I$-fibrant domain, which are weak equivalences, are in the class $\mathcal R$.
\end{sub}

\begin{lem}
A right $I$-fibration is in $\mathcal R$ if and only if it is the dual of a right deformation retract.
\end{lem}

\begin{proof}
We have already seen in Proposition \ref{trivdual}, that morphisms in $\mathcal R$ are duals of deformation retracts. Thus consider a right $I$-fibration $p\colon X \to Y$ which is also the dual of a deformation retract. We have to show that for any morphism $i\colon K \to L\in \mathcal L$ a solution to the lifting problem
\[
\begin{tikzcd}
K\rar{a}\dar[swap]{i} & X\dar{p}\\
L \rar{b}\urar[dashed] & Y	
\end{tikzcd}
\]
exists. Since $p\colon X\to Y$ is the dual of a deformation retract, we have a retraction $s\colon Y\to X$ and a homotopy $h\colon I\otimes X \to X$ from the identity to $sp$. We obtain a solution for the lifting problem
\[
\begin{tikzcd}[column sep = huge, row sep = huge]
I\otimes K \sqcup_{\1\otimes K} \1 \otimes L \rar{(h(id_I\otimes a),sb)}\dar & X\dar{p}\\
I\otimes L \rar{\sigma \otimes b}\urar[dashed]{l} & Y	
\end{tikzcd}
\]
since the right vertical map is a right $I$-anodyne extension. One checks that this solution restricts to a solution of the original lifting problem.
\end{proof}

\begin{prop}\label{rfibtriv}
A right $I$-fibration with right $I$-fibrant codomain is a weak equivalence if and only if it is in the class $\mathcal R$.
\end{prop}

\begin{proof}
Suppose $p\colon X \to Y$ is a right $I$-fibration with right $I$-fibrant codomain, which is also a weak equivalence. We will show that in this case $p$ is the dual of a right deformation retract, hence by the above lemma we may conclude that $p\in \mathcal R$. Since $Y$ is right $I$-fibrant, $p$ is an $I$-homotopy equivalence and we find a map $t\colon Y \to X$ and a homotopy $h\colon I\otimes Y \to Y$ from $id_Y$ to $pt$. Consider the lifting problem
\[
\begin{tikzcd}
\1 \otimes Y \rar{t} \dar & X\dar{p}\\
I\otimes Y \rar{h}\urar[dashed]{h'} & Y
\end{tikzcd}
\]
which admits the indicated lift $h'$ since $p$ is a right $I$-fibration and the left vertical map is a right $I$-anodyne extension. We define $s:= h'_0 \colon Y \to X$. Note that $X$ is right $I$-fibrant and since $p$ is an isomorphism in the $I$-homotopy category and $s$ is a right inverse, and hence inverse to $p$, there is a homotopy $k\colon I\otimes X \to X$ from $id_X$ to $sp$. In general, $k$ does not necessarily exhibit $p$ as a dual of a right deformation retract, since the assumption $pk = \sigma \otimes p$ need not be satisfied. However, we may consider the lifting problem
\[
\begin{tikzcd}
I\otimes \partial I \otimes X \sqcup_{\1 \otimes \partial I \otimes X} \1 \otimes  I \otimes X \arrow{rr}{(k, spk)\cup (\sigma \otimes sp)} \dar & & X\dar{p}\\
I\otimes I\otimes X \rar[swap]{id_I \otimes \sigma} \arrow[dashed]{urr}{K} & I\otimes X \rar[swap]{pk} & Y.
\end{tikzcd}
\]
Now define $k' := K_0\colon I\otimes X \to X$. One readily checks that 
\begin{itemize}
	\item $k'_0= k_0 = id_X$
	\item $k'_1= k_1 = sp$
	\item $pk'= \sigma \otimes pk_0 = \sigma \otimes p$.
\end{itemize}
\end{proof}

\begin{sub}
The only thing left to show to ensure the existence of our desired model structure is that trivial cofibrations are closed under pushouts and transfinite compositions. We will in fact show that they are saturated.
\end{sub}

\begin{lem}\label{fibranttarget}
A morphism in the class $\mathcal L$ with right $I$-fibrant codomain is a weak equivalence if and only if it is a right $I$-anodyne extension.
\end{lem}

\begin{proof}
We already know that right $I$-anodyne extensions are weak equivalences by Proposition \ref{anodyneequivalence}. Thus, let $i\colon K \to L\in \mathcal L$ with right $I$-fibrant codomain. We factorize $i = qj$ where $j$ is right $I$-anodyne and $q$ is a right $I$-fibration. Then $i$ is a weak equivalence if and only if $q$ is. Thus if $i$ is a weak equivalence, it follows from Proposition \ref{rfibtriv} that $q\in \mathcal R$. It follows from the Retract Lemma that $i$ is a retract of $j$, hence a right $I$-anodyne extension.
\end{proof}

\begin{prop}\label{trivcof}
Let $i\colon K\to L\in \mathcal L$. Then $i$ is a weak equivalence if and only if it has the left lifting property with respect to right $I$-fibrations with right $I$-fibrant codomain.
\end{prop}

\begin{proof}
Consider a right $I$-anodyne extension $j\colon L \to L'$, where $L'$ is right $I$-fibrant. If $i$ is a weak equivalence, it follows that $ji$ is a weak equivalence and by the above lemma is in particular a right $I$-anodyne extension. Now consider a diagram
\[
\begin{tikzcd}
K\dar[swap]{i} \rar & X\dar{p}\\
L \rar{f}\dar[swap]{j} & Y \\
L'
\end{tikzcd}
\]
where $p$ is a right $I$-fibration with right $I$-fibrant codomain. Then there exists a lift $\phi\colon L' \to Y$ such that $\phi j = f$. We obtain the diagram
\[
\begin{tikzcd}
K\dar[swap]{ji} \rar & X\dar{p}\\
L' \rar{\phi} & Y 
\end{tikzcd}
\]
which admits a lift since $ji$ is right $I$-anodyne. This lift restricts to a lift of the original diagram.

Conversely, consider a factorization of $ji$ given by
\[
\begin{tikzcd}
K\rar{i}\dar[swap]{k} & L\dar{j}\\
X\rar{p} & L'
\end{tikzcd}
\]
where $k$ is right $I$-anodyne and $p$ a right $I$-fibration. It follows from the Retract Lemma that $ji$ is a retract of $k$, hence a right $I$-anodyne extension. Thus by the 2-out-of-3 property, $i$ is a trivial cofibration.
\end{proof}

\begin{proof}[Proof of Theorem \ref{righthm}]
It follows from Proposition \ref{trivdual} and Proposition \ref{rfibtriv} that a morphism $f\colon A\to B$ is a weak equivalence if and only if there exists a commutative square
\[
\begin{tikzcd}
A\dar[swap]{f} \rar & X\dar \\
B\rar & Y
\end{tikzcd}
\]
in which the horizontal maps are transfinite compositions of pushouts of $\Lambda$ and the right vertical map is a trivial fibration. Furthermore, Proposition \ref{trivcof} implies that the class of trivial cofibrations is saturated, hence Theorem \ref{simpson} guarantees the existence our desired model structure.

Proposition \ref{trivcof} also implies that right $I$-fibrations between right $I$-fibrant objects are fibrations and in particular the fibrant objects are precisely the right $I$-fibrant ones.
\end{proof}

\begin{sub}
We will finish this section with the definition of an \emph{elementary homotopical datum}. Suppose we have fixed an exact functorial cylinder $(I, \partial_0, \partial_1, \sigma)$ on $C$ with respect to $(\mathcal L, \mathcal R)$. 
\end{sub}

\begin{constr}\label{anodynegenerators}
Suppose we have a set of morphisms $S$. Then there is a smallest class of right $I$-anodyne extensions containing $S$, which may be constructed as follows.

Given any set of morphisms $T\subset \mathcal L$, we define the set
\[
\Lambda(T):= \set{\partial I \boxtimes i | i \in T}.
\]
We now choose a generating set $M$ of the class $\mathcal L$ and define the set $\Lambda_I(S,M)$ inductively by setting
\[
\Lambda_I^{0,r} := S \cup \set{\partial_1 \boxtimes i | i \in M}
\]
and
\[
\Lambda^{n+1,r}_I(S,M) := \Lambda \left(\Lambda^{n,r}_I(S,M)\right).
\]
Finally, we define
\[
\Lambda_I^r(S,M) := \bigcup_n \Lambda_I^{n,r}(S,M).
\] 
\end{constr}

\begin{lem}\label{smallanodyne}
The smallest saturated class generated by $\Lambda_I^r(S,M)$ is a class of right $I$-anodyne extensions.
\end{lem}

\begin{proof}
Since $I$ commutes with colimits it has a right adjoint denoted by $(\cdot)^I$. Thus, lifting problems of the form
\[
\begin{tikzcd}
I\otimes K \sqcup_{\1 \otimes K} 1\otimes L\dar \rar & X\dar \\
I\otimes L\urar[dashed] \rar & Y 
\end{tikzcd} 
\] 
correspond to lifting problems of the form
\[
\begin{tikzcd}
K\dar \rar & X^I\dar \\
L \rar\urar[dashed] & X \times_Y Y^I.
\end{tikzcd}
\]
We show that the smallest saturated class containing $\Lambda_I^r(S,M)$ is a class of right $I$-anodyne extensions. The above correspondence shows that whenever $X\to Y$ has the right lifting property with respect to $\Lambda_I^r(S,M)$, then 
\[
X^I \to X\times_Y \times Y^I
\]
has the right lifting property with respect to any morphism in $M$ and hence any morphism in $\mathcal L$. Thus, the saturated class is closed under the operation $\partial_1 \boxtimes (\cdot)$. A similar argument shows that it is also closed under the operation $\partial I \boxtimes (\cdot)$. Conversely, it is clear that any class of right $I$-anodyne extensions which contains $S$ is contained in the weakly saturated class generated by $\Lambda_I^r(S,M)$.
\end{proof} 

\begin{defn}
An \emph{elementary homotopical datum} consists of an exact cylinder $(I,\partial_0, \partial_1,\sigma)$ together with a set of morphisms $S$.
\end{defn}

\begin{Example}\label{segmentgenerators}
Let $A$ be a small category and consider the exact cylinder $I\times (\cdot)$ on $\psh(A)$ as in Example \ref{separatingsegment}. Consider the elementary homotopical datum given by $(I,\emptyset)$. Then the right $I$-anodyne extensions have a particularly simple description. Let $M$ be a cellular model for $\psh(A)$ and consider the set of morphisms
\[
I \times K \sqcup_{\1 \times K} \1 \times L \to I \times L
\]
for $K \to L \in M$. An easy calculation shows that the saturated class generated by this set is the class of right $I$-anodyne extensions associated to $(I,\emptyset)$, see also \cite[Remarque 1.3.15]{cisinski}.
\end{Example}

\begin{sub}
Let us consider an elementary homotopical datum given by $(I,\partial_0, \partial_1,\sigma)$ and $S$, which we will denote by $(I,S)$ for brevity. Then by Lemma \ref{smallanodyne} and its dual version, we obtain a right as well as a left homotopical structure and hence by Theorems \ref{righthm} and \ref{lefthm} two model structures on the category of presheaves on $C$.
\end{sub}

\begin{defn}
Let $(I,S)$ be an elementary homotopical structure and denote by $r(I,S)$ the right homotopical structure generated by it and by $l(I,S)$ the left homotopical structure generated by it. We will call the model structure induced by $r(I,S)$ the \emph{Contravariant model structure generated by $(I,S)$} and the model structure induced by $l(I,S)$ the \emph{Covariant model structure generated by $(I,S)$}. An equivalence in the Contravariant model structure is called a \emph{contravariant equivalence} and an equivalence in the Covariant model structure is called a \emph{covariant equivalence}.
\end{defn}

\section{Abstract finality}\label{abstrcof}

\begin{sub}
Suppose we have an elementary homotopical datum. In the previous section we have established two model structures arising from such a datum, the Covariant and the Contravariant model structure. In this section, we will discuss the notions of final and initial maps, which arise very naturally in this setting. To this end, we will consider Co- and Contravariant model structures for \emph{families}.
\end{sub}

\begin{constr}\label{relhomotopicaldatum}
Suppose we have an object $A\in C$. Then we have a weak factorization system $(\mathcal L_A, \mathcal R_A)$ on $C/A$, where the left class is defined using the forgetful functor to $C$. The cylinder $I$ on $C$ induces a cylinder $I_A$ on the category of $C/A$ whose action on objects $p\colon X \to A$ is given by the composition
\[
I\otimes X \xrightarrow{\sigma} X \xrightarrow{p} A.
\]
Starting with a class of right $I$-anodyne extensions $An^r(I)$, it is easy to check that the class $An^r(I_A)$ of those morphisms in $C/A$, whose underlying maps in $C$ are right $I$-anodyne extension, defines a class of right $I_A$-anodyne extensions.
\end{constr}

\begin{sub}
Thus applying Theorem \ref{righthm}, we obtain a relative version.
\end{sub}

\begin{thm}\label{relativerighthm}
There exists a unique model structure on the category $C/A$ with cofibrations the class $\mathcal L_A$ and fibrant objects the right $I$-fibrations with target $A$. Dually, there exists a unique model structure on $C/A$ with cofibrations the class $\mathcal L_A$ and fibrant objects the left fibrations with target $A$.
\end{thm}

\begin{sub}
Now fix an elementary homotopical datum $\mathcal I:=(I,S)$. By the above theorem, we obtain for any object $A\in C$ a Contravariant and Covariant model structure on the category of $C/A$ induced by $\mathcal I$.
\end{sub}

\begin{defn}
A morphism $f\colon X\to Y$ is called $\mathcal I$\emph{-final} if for all objects $A$ and all morphisms $p\colon Y \to A$ the induced morphism
\[
\begin{tikzcd}
X \rar{f} \drar[swap]{p\circ f} & Y \dar{p}\\
 & A
\end{tikzcd}
\]
is a contravariant equivalence in the category $C/A$. Dually, it is called \emph{$\mathcal I$-initial} if the above morphism is a covariant equivalence in the category $C/A$.
\end{defn}

\begin{Remark}
There is conflicting terminology in the literature. What we have called \emph{final} agrees with Cisinski's definition in \cite{cisinskibook} and Joyal's definition in \cite{joyalnotes}, \cite{joyalquadern}, while Lurie calls these morphisms \emph{cofinal} in \cite{lurie}. On the other hand, what we have called \emph{initial} is called \emph{cofinal} in \cite{cisinskibook}, while it is not explicitly defined in \cite{lurie}.
\end{Remark}

\begin{sub}
Thus, the $\mathcal I$-final (resp. $\mathcal I$-initial) maps are precisely those, which are equivalences in the contravariant (resp. covariant) model structures for \emph{all} families. 
\end{sub}

\begin{lem}\label{finalfibrations}
A morphism in the class $\mathcal L$ is $\mathcal I$-final if and only if it is a right $\mathcal I$-anodyne extension. A right $\mathcal I$-fibration is $\mathcal I$-final if and only if it is in the class $\mathcal R$.
\end{lem}

\begin{proof}
By construction it is clear that right $\mathcal I$-anodyne extensions are $\mathcal I$-final. Conversely, if $i\colon X\to Y$ is $\mathcal I$-final then it is in particular a trivial cofibration with fibrant domain in Contravariant model structure on $C/Y$. By Lemma \ref{fibranttarget}, it is right $\mathcal I$-anodyne.

It is also clear that any map in the class $\mathcal R$ is $\mathcal I$-final. Conversely, if $p\colon X \to Y$ is a right $\mathcal I$-fibration which is also $\mathcal I$-final, then it is a right $\mathcal I$-fibration with fibrant codomain in the Contravariant model structure on $C/Y$ which is a weak equivalence. By Proposition \ref{rfibtriv} it is in the class $\mathcal R$.
\end{proof}

\begin{prop}
The class of $\mathcal I$-final maps satisfies the right cancellation property.
\end{prop}

\begin{proof}
Suppose we have a composable sequence
\[
X\xrightarrow{f} Y \xrightarrow{g} Z
\]
and assume that $f$ is $\mathcal I$-final. Consider any morphism $Z\to A$, and consider $g$ and $gf$ as morphisms in $C/A$. Then by the 2-out-of-3 property of weak equivalences is is clear that $g$ is a Contravariant equivalence in $C/A$ if and only if $gf$ is. Thus the $\mathcal I$-final maps satisfy the right cancellation property.\end{proof}

\begin{sub}
The next Proposition shows, that the $\mathcal I$-final and $\mathcal I$-initial maps are completely determined by the elementary homotopical datum $\mathcal I$.
\end{sub}

\begin{prop}\label{leftcancel}
A map is $\mathcal I$-final if and only if it can be factorized as a right $\mathcal I$-anodyne extension followed by a map in the class $\mathcal R$.
\end{prop}

\begin{proof}
By the right cancellation property the class of $\mathcal I$-final maps is closed under composition. By Lemma \ref{finalfibrations} both right $\mathcal I$-anodyne extensions and maps in the class $\mathcal R$ are $\mathcal I$-final, hence their composition is $\mathcal I$-final. Conversely, suppose $f$ is an $\mathcal I$-final map. We may factorize $f=pi$ with $i$ a right $\mathcal I$-anodyne extension and $p$ a right $\mathcal I$-fibration. By the right cancellation property, $p$ is $\mathcal I$-final thus by Lemma \ref{finalfibrations} we have $p\in \mathcal R$.
\end{proof}

\begin{defn}
Let $p\colon X \to Y$ be a morphism and consider a diagram of the form
\[
\begin{tikzcd}
A' \rar{j}\dar & B' \rar \dar & X\dar{p}\\
A\rar{i} & B \rar & Y
\end{tikzcd}
\]
in which the squares are cartesian. Then $p$ is called \emph{$\mathcal I$-proper} if $j$ is $\mathcal I$-final whenever $i$ is $\mathcal I$-final. Dually, $p$ is called \emph{$\mathcal I$-smooth} if $j$ is $\mathcal I$-initial whenever $i$ is $\mathcal I$-initial.
\end{defn}

\begin{Remark}
Again, there is conflicting terminology in the literature. What we have called \emph{proper} agrees with the definitions of Cisinski \cite{cisinskibook} and Joyal \cite{joyalnotes}, \cite{joyalquadern} in the example of simplicial sets, while Lurie calls these morphisms \emph{smooth} in \cite{lurie}.
\end{Remark}

\begin{sub}
In some cases, right (resp. left) $\mathcal I$-fibrations provide examples of $\mathcal I$-smooth (resp. $\mathcal I$-proper) maps. Although it is not true general, that they are $\mathcal I$-smooth (resp. $\mathcal I$-proper), there is a particular class of left (resp. right) $\mathcal I$-anodyne extensions, which are always preserved by pullback along a right (resp. left) $\mathcal I$-fibration.
\end{sub}

\begin{lem}\label{deformright}
Any right deformation retract is a right $\mathcal I$-anodyne extension and any left deformation retract is a left $\mathcal I$-anodyne extension.
\end{lem}

\begin{proof}
We only show the case of a right deformation retract. Thus, let $i\colon K\to L$ be a right deformation retract with retraction $r\colon L\to K$ and homotopy $h\colon I\otimes L \to L$ from $id_Y$ to $ir$ which is constant on $K$. We obtain a commutative diagram
\[
\begin{tikzcd}[column sep = huge]
K\rar{\partial_0} \dar[swap]{i}	 & I\otimes K \sqcup_{\1\otimes K} \1 \otimes L \rar{(\sigma, r)} \dar & K\dar{i} \\
L \rar{\partial_0} & I\otimes L \rar{h} & L
\end{tikzcd}
\]	
exhibiting $i$ as a retract of a right $\mathcal I$-anodyne extension.
\end{proof}

\begin{prop}\label{leftright}
Consider a Cartesian square
\[
\begin{tikzcd}
A\rar{j}\dar & X\dar{p}\\
B \rar{i} & Y.
\end{tikzcd}
\]
If $i$ is a right deformation retract and $p$ is a left $\mathcal I$-fibration, then $j$ is a right deformation retract. Dually, if $i$ is a left deformation retract and $p$ is a right $\mathcal I$-fibration, then $j$ is a left deformation retract.
\end{prop}

\begin{proof}
We only show the case when $p$ is a left $\mathcal I$-fibration and $i$ is a right deformation retract. Suppose we have a retraction $r\colon Y \to B$ and a homotopy $h\colon I\otimes Y \to Y$ from $id_Y$ to $ir$ which is constant on $B$. We obtain a solution $k$ in the following lifting problem
\[
\begin{tikzcd}[column sep = huge]
I\otimes A \sqcup_{\{0\}\otimes A} \{0\}\otimes X \rar{(\sigma \otimes j, id_X)}\dar & X\dar{p}\\
I\otimes X \rar{h(id_I \otimes p)}\urar[dashed]{k} & Y
\end{tikzcd}
\]
since $p$ is a left $\mathcal I$-fibration and the left vertical map is left $\mathcal I$-anodyne. We claim that $k$ exhibits $j$ as a deformation retract. We have a map
\[
X\xrightarrow{p} Y \xrightarrow{r} B \xrightarrow{i} Y
\]
and also
\[
X\xrightarrow{k_1} X \xrightarrow{p} Y.
\]
Since $pk_1=irp$ we get a unique map $s\colon X \to A$. Now we have $jsj = k_1 j = j$ and $qsj = rpj=q$ hence $sj= id_A$. Finally one checks that the homotopy $k$ satisfies the right properties.
\end{proof}

\section{Examples}\label{exa}

\begin{sub}
We will consider two examples in this section, the Covariant and Contravariant model structures for simplicial sets and the coCartesian and Cartesian model structures for marked simplicial sets. We will only consider the Contravariant and Cartesian model structures and we will show that both are examples of Contravariant model structures arising from an elementary homotopical datum as introduced in this article.

The Contravariant model structure for simplicial sets is originally due to Joyal, see for example \cite{joyalnotes}, and is obtained using purely combinatorial methods. Lurie gives an alternative construction of these model structures in \cite{lurie}, using a comparison to simplicial categories. Another approach is in Cisinski's book \cite{cisinskibook}, using his theory of \emph{anodyne extensions} which is our starting point.

The coCartesian model structure for marked simplicial sets was introduced by Lurie in \cite{lurie}. We give a new proof for its existence using our theory of Contravariant model structures. In particular, we obtain a description of fibrations between fibrant objects for the coCartesian model structure.
\end{sub}

\subsection{The Contravariant model structure for simplicial sets}

\begin{sub}
Consider the representable simplicial set $\Delta^1$. Again, the inclusion of the endpoints is disjoint, hence by Example \ref{separatingsegment} we obtain the exact cylinder
\[
\Delta^1 \times (\cdot) \colon \mathbf{sSet}\to \mathbf{sSet}
\]
with respect to (Mono,Triv). Consider the elementary homotopical datum $\mathcal I := (\Delta^1, \emptyset)$. By Example \ref{segmentgenerators} the right $\mathcal I$-anodyne extensions are precisely the saturated class generated by
\[
\Delta^1 \times \partial \Delta^n \cup \1 \times \Delta^n \to \Delta^1 \times \Delta^n
\]
for $n \geq 0$. In fact, this is a familiar class.
\end{sub}

\begin{lem}
The following sets of morphisms generate the same saturated class.
\begin{enumerate}
	\item $\Delta^1 \times \partial \Delta^n \cup \1 \times \Delta^n \to \Delta^1 \times \Delta^n$ for $n\geq 0$,
	\item $\Lambda^n_k \to \Delta^n$ for $0<k\leq n$.
\end{enumerate}
\end{lem}

\begin{proof}
See for example \cite[Lemma 3.1.3]{cisinskibook}, [Lurie].
\end{proof}

\begin{sub}
In other words, our right $\mathcal I$-anodyne extensions are precisely the right anodyne extensions of simplicial sets. Dually, the left $\mathcal I$-anodyne extensions are precisely the left anodyne extensions of simplicial sets. We obtain Covariant and Contravariant model structures for the elementary homotopical datum given by $\mathcal I$. In this case, the model structures in families will be important for us. Applying Theorem \ref{relativerighthm} we obtain the following.
\end{sub}

\begin{thm}\label{covariant}
Let $A$ be a simplicial set. There exists a unique model structure on $\mathbf{sSet}/A$ with cofibrations the monomorphisms and fibrant objects the right fibrations of simplicial sets with target $A$. Dually, there exists a unique model structure on $\mathbf{sSet}/A$ with cofibrations the monomorphisms and fibrant objects the left fibrations of simplicial sets with target $A$.
\end{thm}

\begin{sub}
Thus, we recover Joyal's Covariant and Contravariant model structures.
\end{sub}

\begin{defn}
We will denote the Contravariant model structure above by $\mathbf{RFib}(A)$ and the Covariant model structure by $\mathbf{LFib(A)}$.
\end{defn}

\begin{Remark}
It is possible to construct the Co- and Contravariant model structure using anodyne extensions, see Remark \ref{plainanodyne}, instead of left and right ones. We can consider the cylinder $J\times (\cdot)$ on simplicial sets, where $J$ is the nerve of the free walking isomorphism, and the elementary homotopical datum given by $J$ and the outer horn inclusions
\[
\Lambda^n_k \to \Delta^n 
\]
where $0<k\leq n$. One then has to check that the class of anodyne extensions associated to this elementary homotopical datum is indeed the class of right anodyne extensions. This is carried out in \cite[Chapter 4]{cisinskibook}.
\end{Remark}

\begin{sub}
The previous section shows that we obtain abstractly a notion of final and initial maps. We translate this to the following definition.
\end{sub}

\begin{defn}
A map of simplicial sets $f\colon X \to Y$ is \emph{final} if and only if for all simplicial sets $A$ and all maps $Y\to A$, the morphism $f$ is a Contravariant equivalence in $\mathbf{RFib}(A)$. Dually, the map $f$ is called \emph{initial} if it induces a Covariant equivalence in $\mathbf{LFib}(A)$.
\end{defn}

\begin{sub}
We may also consider smooth and proper maps in this setting. Recall from the previous section that a map $p\colon X\to Y$ is proper if and only if for any diagram
\[
\begin{tikzcd}
A' \rar{w} \dar & B'\rar\dar & X\dar{p}\\
A\rar{v} & B \rar& Y
\end{tikzcd}
\]
in which the squares are pullbacks, the map $w$ is final if $v$ is final. Dually, $p$ is called smooth if the map $w$ is initial whenever $v$ is initial.
\end{sub}

\begin{prop}\label{properleft}
Left fibrations are proper and right fibrations are smooth.
\end{prop}

\begin{proof}
We only show the case of left fibrations. Since left fibrations are closed under pullback, it suffices to show that for any cartesian square
\[
\begin{tikzcd}
A'\rar{w}\dar & X\dar{p}\\
A\rar{v} & Y
\end{tikzcd}
\]
in which $v$ is final and $p$ is a left fibration, the map $w$ is final. Since any final map can be factorized as a right anodyne map followed by a trivial fibration and trivial fibrations are closed under pullback, it suffices to show that $w$ is right anodyne whenever $v$ is right anodyne. Let $\mathcal A$ be the class of morphisms whose pullbacks are right anodyne. Then this class is saturated and satisfies the right cancellation property, since this is true for right anodyne extensions. Thus it suffices to show that $\mathcal A$ contains the class of right anodyne extensions, and hence it suffices to show the assertion when $v$ is an element of the generating set of right anodyne extension. We have already seen that the right anodyne extensions are the saturated class generated by 
\[	
\Delta^1 \times \partial \Delta^n \cup \1 \times \Delta^n \to \Delta^1 \times \Delta^n
\]
for $n\geq 0$. By the right cancellation property, it suffices to show the assertion for morphisms of the form
\[
\1 \times K \to \Delta^1 \times K
\]
for any simplicial set $K$. We now observe that the above map is a right deformation retract, hence by Proposition \ref{leftright} its pullback along any left fibration is a right deformation retract, hence right anodyne.
\end{proof}

\subsection{The Cartesian model structure}

\begin{sub}
We list some basic properties of Cartesian fibrations from \cite{lurie} and recall some basic facts about the category of marked simplicial sets.
\end{sub}

\begin{prop}\label{cartedge}
Let $p\colon X \to A$ be an inner fibration of simplicial sets and let $f\colon x\to y \in X$ be an edge. Then the following are equivalent.

\begin{enumerate}
	\item The induced map
	\[
	X_{/f} \to X_{/x}\times_{A_{/p(x)}} A_{/p(f)} 
	\]
	is a trivial fibration.
	\item For all $n \geq 2$ and all lifting problems of the form
	\[
	\begin{tikzcd}
	\Delta^{\{n-1,n\}}\drar{f}\dar & \\
	\Lambda^n_n \rar \dar & X\dar{p}\\
	\Delta^n \rar \urar[dashed] & A.
	\end{tikzcd}
	\]
	there exists a lift as indicated.
	\item For all $n \geq 1$ and all lifting problems of the form
	\[
	\begin{tikzcd}
	\Delta^1\times \{1\} \dar \drar{f} & \\
	\Delta^1 \times \partial \Delta^n \cup \{1\} \times \Delta^n \rar \dar & X\dar{p}\\
	\Delta^1 \times \Delta^n \rar\urar[dashed] & A
	\end{tikzcd}
	\]
	there exists a lift as indicated.
\end{enumerate}
\end{prop}

\begin{proof}
Combine \cite[Definition 2.4.1.1]{lurie}, \cite[Remark 2.4.1.4]{lurie} and \cite[Proposition 2.4.1.8]{lurie}.
\end{proof}

\begin{defn}
Let $p\colon X\to A$ be an inner fibration. Then an edge $f\colon \Delta^1 \to X$ is called \emph{p-Cartesian} if it satisfies the equivalent conditions of the above Proposition.
\end{defn}

\begin{prop}\label{cartiso}
Let $p\colon C\to D$ be an inner fibration between $\infty$-categories and let $f\colon \Delta^1 \to C$ be an edge. Then the following are equivalent.
\begin{enumerate}
\item The edge $f$ is an equivalence in $C$.
\item The edge $f$ is $p$-Cartesian and its image $p(f)$ is an equivalence in $D$.
\end{enumerate}
\end{prop}

\begin{proof}
See \cite[Proposition 2.4.1.5]{lurie}.
\end{proof}

\begin{prop}\label{cancelcomp}
Let $p\colon X \to A$ be an inner fibration between $\infty$-categories. Let $\sigma\colon \Delta^2 \to X$ be a 2-simplex depicted as
\[
\begin{tikzcd}
{} & \cdot \dar{g}\\
\cdot \urar{f}\rar[swap]{h} & \cdot 
\end{tikzcd}
\]
Suppose that the edge $g$ is $p$-Cartesian. Then $f$ is $p$-Cartesian if and only if $h$ is $p$-Cartesian.
\end{prop}

\begin{proof}
See \cite[Proposition 2.4.1.7]{lurie}.
\end{proof}

\begin{defn}
Let $p\colon X\to A$ be an inner fibration. Then $p$ is called a \emph{Cartesian fibration} if for all lifting problems of the form
\[
\begin{tikzcd}
\Delta^{\{1\}} \rar \dar & X\dar{p}\\
\Delta^1 \rar \urar[dashed] & A,
\end{tikzcd}
\]
there exists a lift as indicated, which is $p$-Cartesian.
\end{defn}

\begin{sub}
We denote by $\mathbf{sSet}^+$ the category of \emph{marked simplicial sets}. It's objects are pairs $(A,E_A)$ where $A$ is a simplicial set and $E_A\subseteq A_1$ is a collection of 1-simplices of $A$ containing all degenerate 1-simplices, which are called the \emph{marked edges}. The morphisms are given by morphisms of simplicial sets which preserve the marked edges.
\end{sub}

\begin{sub}
The category $\mathbf{sSet}^+$ is locally cartesian closed. Moreover, colimits are universal and coproducts are disjoint (see \cite{verity}). 
\end{sub}

\begin{sub}
Consider the class of morphisms of marked simplicial sets, whose underlying morphism of simplicial sets is a monomorphism. It is easy to see that this class is generated as a saturated class by the set of morphisms given by
\begin{itemize}
	\item $(\partial \Delta^n)^\flat\to (\Delta^n)^\flat$ for $n\geq 0$,
	\item $(\Delta^1)^\flat \to \mcyl$.
\end{itemize}{}
Thus, we obtain a cofibrantly generated weak factorization system with left class the maps with underlying monomorphisms and whose right class we refer to as \emph{trivial fibrations}. We denote this factorization system also as (Mono, Triv).
\end{sub}

\begin{sub}
The forgetful functor
\[
\mathbf{sSet}^+ \to \mathbf{sSet}
\]
has both a left and a right adjoint. We denote the left adjoint by 
\[
(\cdot)^\flat \colon \mathbf{sSet} \to \mathbf{sSet}^+
\]
Given a simplicial set $A$, the marked simplicial set $A^\flat$ has precisely the degenerate 1-simplices marked. The right adjoint will be denoted by
\[
(\cdot)^\sharp \colon \mathbf{sSet}\to \mathbf{sSet}^+
\]
Given a simplicial set $B$, the marked simplicial set $B^\sharp$ has all 1-simplices marked. 
\end{sub}

\begin{sub}
The functor $(\cdot)^\sharp$ has a further right adjoint, denoted by 
\[
\mu \colon \mathbf{sSet}^+\to \mathbf{sSet}
\]
Given a marked simplicial set $(A,E_A)$, the simplicial set $\mu(A,E_A)$ is the simplicial subset of $A$ spanned by the marked edges.
\end{sub}

\begin{sub}
Suppose $p\colon X\to A$ is an inner fibration of simplicial sets. We will denote by $X^\natural$ the marked simplicial set, which has precisely the $p$-Cartesian edges marked. Note that this is slightly abusive notation, since the marked simplicial set $X^\natural$ depends on the map of simplicial sets $p$.
\end{sub}

\begin{sub}
Consider the functor
\[
\mcyl \times (\cdot)\colon \mathbf{sSet}^+ \to \mathbf{sSet}^+
\]
We claim that this is an exact cylinder for marked simplicial sets.
\end{sub} 

\begin{lem}
The cylinder
\[
\mcyl \times (\cdot)\colon \mathbf{sSet}^+\to \mathbf{sSet}^+
\]
is an exact cylinder with respect to the factorization system \emph{(Mono, Triv)} on $\mathbf{sSet}^+$.
\end{lem}

\begin{proof}
It is clear that $\mcyl \times (\cdot)$ preserves colimits, since $\mathbf{sSet}^+$ is cartesian closed. Let $i\colon (K,E_K)\to (L,E_L)$ be a monomorphism. Consider the commutative diagram
\[
\begin{tikzcd}
\partial \mcyl \times (K,E_K) \rar \dar & \partial \mcyl \times (L,E_L)\dar\arrow[bend left]{ddr} & \\ 
\mcyl \times (K,E_K)\rar\arrow[bend right]{drr} &\mcyl \times (K,E_K)\cup\partial \mcyl \times (L,E_L)\drar  & \\
& & \mcyl \times (L,E_L)
\end{tikzcd}
\]
in which the square is a pushout. We need to show that the map $\partial \mcyl \boxtimes i$ is in the class $\mathcal L$, i.e. its underlying map of simplicial sets is a monomorphism. But the forgetful functor to simplicial sets is a left adjoint, hence the above diagram gives a diagram in simplicial sets in which the square is a pushout. Since the cylinder $\Delta^1\times(\cdot)$ is exact, it follows that the underlying map of $\partial \mcyl \boxtimes i$, which is just $\partial \Delta^1 \boxtimes i$, is a monomorphism. The same argument holds for $\partial_j \boxtimes i$ for $j=0,1$.
\end{proof}

\begin{defn}
Let $\mathcal I^+$ be the elementary homotopical datum associated to the exact cylinder $\mcyl \times (\cdot)$ with respect to (Mono, Triv), and the set of maps defined by
\begin{itemize}
	\item $(\Lambda^n_k)^\flat \to (\Delta^n)^\flat$ for $n \geq 2$ and $0<k<n$,
	\item $J^\flat \to J^\sharp$ where $J$ is the nerve of the free walking isomorphism.
\end{itemize}
\end{defn}

\begin{sub}
Now let $(A,E_A)$ be a marked simplicial set. By Construction \ref{relhomotopicaldatum} we obtain an elementary homotopical datum $\mathcal I^+_{(A,E_A)}$. Thus applying Theorem \ref{relativerighthm} we obtain a Contravariant and Covariant model structure on $\mathbf{sSet}^+$.
\end{sub}

\begin{thm}\label{markedcov}
For any marked simplicial set $(A,E_A)$, there is a Contravariant and Covariant model structure on $\mathbf{sSet}^+/(A,E_A)$ induced by $\mathcal I^+_{(A,E_A)}$. 
\end{thm}

\begin{defn}
We will call the Contravariant model structure on $\mathbf{sSet}^+/(A,E_A)$ the \emph{Cartesian model structure}. We denote this model category by $\mathbf{Cart}(A,E_A)$. We will refer to the right $\mathcal I^+_{(A,E_A)}$-anodyne extensions as \emph{marked right anodyne extensions} and to the right $\mathcal I^+_{(A,E_A)}$-fibrations \emph{marked right fibrations.} Furthermore, we refer to the weak equivalences as \emph{Cartesian equivalences}.

Dually, we will call the Covariant model structure on $\mathbf{sSet}^+/(A,E_A)$ the \emph{coCartesian model structure}. We denote this model category by $\mathbf{coCart}(A,E_A)$. We will refer to the left anodyne extensions and left fibrations as \emph{marked left anodyne extensions} and \emph{marked left fibrations.} Furthermore, we refer to the weak equivalences as \emph{coCartesian equivalences}.
\end{defn}

\begin{sub}
The rest of this section only considers the Cartesian model structure. The associated statements for the coCartesian model structure easily follow by duality. We first compare our model structures to Lurie's, thus would like to have a finer control on the marked right anodyne extensions and marked right fibrations. To this end we first construct more explicit generators for the marked right anodyne extensions.
\end{sub} 

\begin{defn}\label{markedgenerators}
We define $\mathcal A$ to be the smallest saturated class containing the morphisms
\begin{itemize}
	\item[(A1)] $(\Lambda^n_k)^\flat \to (\Delta^n)^\flat$ for $n \geq 2$ and $0<k<n$,
	\item[(A2)] $J^\flat \to J^\sharp$, 
	\item[(B1)] $\mcyl \times (\Delta^1)^\flat \cup \1\times \mcyl \to \mcyl \times \mcyl$,
	\item[(B2)] $\mcyl \times (\partial \Delta^n)^\flat\cup \1 \times (\Delta^n)^\flat \to \mcyl \times (\Delta^n)^\flat$.
\end{itemize}
\end{defn}

\begin{lem}\label{productclosure}
For all monomorphisms $K\to L$ and all $A\to B \in \mathcal A$ the morphism
\[
A \times L \cup B\times K \to L \times B
\]
is also in $\mathcal A$.
\end{lem}

\begin{proof}
It suffices to show this for the generators (A1) and (A2). Recall that the monomorphisms in $\mathbf{sSet}^+$ are generated by the morphisms
\begin{enumerate}
	\item $(\Delta^1)^\flat\to \mcyl$
	\item $(\partial \Delta^n)^\flat \to (\Delta^n)^\flat$.
\end{enumerate}
We observe that the pushout product of (A1) and (1) as well as the pushout product of (A2) and (2) will yield isomorphisms. The pushout product of (A1) and (2) is an inner anodyne extension of simplicial sets and hence in $\mathcal A$. It remains to show that the pushout product 
\[
(\Delta^1)^\flat \times J^\sharp \cup \mcyl \times J^\flat \to \mcyl \times J^\sharp 
\]
is in $\mathcal A$. We observe that this map is an iterated pushout of maps in the class (B1).
\end{proof}

\begin{lem}
The class $\mathcal A$ is the class of marked right anodyne extensions.
\end{lem}

\begin{proof}
It is clear that the class of marked right anodyne extensions contains the class $\mathcal A$. Conversely, Construction \ref{anodynegenerators} gave an explicit generating set for marked right anodyne extensions. Recall that this set of generators was constructed inductively and in our situation this takes the following form. The starting set is given by the set
\[
(A1)\cup (A2)\cup \set{\partial_1 \boxtimes i | i \in (B1)\cup (B2)}.
\]
We observe that the morphisms $\partial_1 \boxtimes i$ above are all in $\mathcal A$ by Lemma \ref{productclosure}, hence the above set is contained in $\mathcal A$. To finish the proof we observe that, in the notation of Construction \ref{anodynegenerators}, we have $\Lambda(\mathcal A)\subseteq \mathcal A$ again by Lemma \ref{productclosure}. Thus, any morphism in the generating set for marked right anodyne extensions is in $\mathcal A$ and hence $\mathcal A$ contains the marked right anodyne extensions.
\end{proof}

\begin{lem}\label{kan}
For any $\infty$-groupoid $K$, the morphism $K^\flat \to K^\sharp$ is a marked right anodyne extension.
\end{lem}

\begin{proof}
We have a pushout diagram
\[
\begin{tikzcd}
\sqcup \ J^\flat \rar \dar & K^\flat\dar \\
\sqcup \ J^\sharp \rar & K^\sharp 
\end{tikzcd}
\]
where the coproduct is taken over all possible maps $J \to K$.
\end{proof}

\begin{sub}
The following proposition characterizes the marked right fibrations of 	$\mathbf{sSet}^+$ in the Cartesian model structure. Its proof is adapted from \cite[Proposition 3.1.1.6]{lurie}.
\end{sub}

\begin{prop}
Let $p\colon (X,E_X)\to (A,E_A)$ be a morphism of marked simplicial sets. Then $p$ is a marked right fibration if and only if the following conditions hold.
\begin{enumerate}
	\item The underlying map of simplicial sets is an inner fibration.
	\item For any $y\in X$ and any marked $\overline f\colon x \to p(y)\in E_A$, there exists a marked edge $f\in E_X$ such that $p(f)= \overline f$.
	\item An edge $f\colon \Delta^1 \to X$ is marked if and only if $p(f)\in E_A$ and $f$ is $p$-Cartesian.
\end{enumerate}
\end{prop}

\begin{proof}
We first show the `only if' direction. Thus suppose the map $p\colon(X,E_X)\to (A,E_A)$ is a marked right fibration. Hence it satisfies the right lifting property with respect to the generators of Definition \ref{markedgenerators}. The right lifting property with respect to (A1) implies that $p$ is an inner fibration. The right lifting property with respect to (B2) for $n=0$ implies that over each marked edge of the form $\overline f \colon x \to p(y)\in E_A$ there exists a marked edge $f\in E_X$ such that $p(f)=\overline f$. Moreover, the right lifting property with respect to (B2) for $n\geq 1$ shows that every marked edge is $p$-Cartesian by Proposition \ref{cartedge}. It remains to show that an edge is marked only if it is $p$-Cartesian and its image is marked.

Suppose we have an edge $f\colon x\to y$ such that that $p(f)$ is marked and $f$ is $p$-Cartesian. We have already seen that there exists a marked edge $f'\colon x'\to y$ such that $p(f')=p(f)$ and $f'$ is $p$-Cartesian. In particular, we find a 2-simplex in $X$ of the form
\[
\begin{tikzcd}
{} & x'\dar{f'}\\
x\rar[swap]{f}\urar{\alpha} & y.
\end{tikzcd}
\]
Since $f$ is also $p$-Cartesian, it follows by the right cancellation property of Cartesian edges, Lemma \ref{cancelcomp}, that $\alpha$ is $p$-Cartesian. In particular, the edge $\alpha$ defines an equivalence in the fiber $X_{p(x)}$, which is an $\infty$-category. Consider the maximal $\infty$-groupoid $k(X_{p(x)})$. Since the map
\[
X_{p(x)}\to \ast 
\] 
is a marked right fibration, it has the right lifting property with respect to the map
\[
k(X_{p(x)})^\flat \to k(X_{p(x)})^\sharp, 
\]
which implies that every equivalence of $k(X_{p(x)})$ and thus in particular $\alpha$ is marked. By the right lifting property with respect to (B1), it follows that $f$ is also marked.

Now assume that $p\colon (X,E_X)\to (A,E_A)$ satisfies the assumptions of the proposition. We show that $p$ is a marked right fibration. Thus we need to show the right lifting property against the generators of Definition \ref{markedgenerators}. The right lifting property against (A1) follows since $p$ is an inner fibration. 

To show the right lifting property against (A2) it suffices to consider the case where $(A,E_A)= J^\sharp$. In this case $p$ is an inner fibration over a Kan complex, hence the $p$-Cartesian edges are precisely the equivalences by Lemma \ref{cartiso}, thus $p$ has the right lifting property against (A2).

The right lifting property property against (B1) follows immediately from assumptions (2) and (3) and the right lifting property against (B2) follows since $p$-Cartesian edges satisfy the right cancellation property, by Lemma \ref{cancelcomp}.
\end{proof}

\begin{coro}\label{cocartesianfibrant}
Let $X\to A$ be an inner fibration of simplicial sets. Then it is a Cartesian fibration if and only if the map $X^\natural \to A^\sharp$ is a marked right fibration.
\end{coro}

\begin{coro}\label{absolutefibrant}
A marked simplicial set $(X,E_X)$ is marked right fibrant if and only if $X$ is an $\infty$-category and $E_X$ is precisely the set of equivalences in $X$.
\end{coro}

\begin{sub}
Having established the (co)Cartesian model structure for marked simplicial sets, we give a new proof that coCartesian fibrations are proper (see \cite[Proposition 4.1.2.5]{lurie} for an alternative proof). As preparation, we relate marked right anodyne extensions and right anodyne extensions.
\end{sub}

\begin{lem}\label{underlyingmarked}
Let $(K,E_K)\to (L,E_L)$ be a marked right anodyne extension. Then the underlying map of simplicial sets is a right anodyne extension.
\end{lem}

\begin{proof}
Suppose $X\to A$ is a right fibration of simplicial sets and consider a lifting problem
\[
\begin{tikzcd}
K\rar \dar & X\dar \\
L \rar\urar[dashed] & A
\end{tikzcd}
\]
Since $X\to A$ is a right fibration, the map $X^\sharp \to A^\sharp$ is a marked right fibration by Corollary \ref{cocartesianfibrant} and we obtain a lifting problem
\[
\begin{tikzcd}
(K,E_K)\rar \dar & X^\sharp\dar \\
(L,E_L) \rar\urar[dashed] & A^\sharp
\end{tikzcd}
\]
which admits a solution since $(K,E_K)\to (L,E_L)$ was assumed to be marked right anodyne, thus providing a solution to the original lifting problem.
\end{proof}

\begin{lem}\label{markedright}
Let $K\to L$ be a right anodyne extension of simplicial sets. Then $K^\sharp \to L^\sharp$ is a marked right anodyne extension.
\end{lem}

\begin{proof}
Since $(\cdot)^\sharp$ commutes with colimits, it suffices to show that the image of the generators of right anodyne extensions are marked right anodyne. Thus we need to show that
\[
\mcyl \times (\partial \Delta^n)^\sharp \cup \1 \times (\Delta^n)^\sharp \to \mcyl \times (\Delta^n)^\sharp 
\]
is marked right anodyne, which follows from Lemma \ref{productclosure}, since $(\partial \Delta^n)^\sharp \to (\Delta^n)^\sharp$ is a monomorphism.
\end{proof}

\begin{defn}
A marked right anodyne extension is called \emph{cellular} if it is in the smallest saturated class generated by
\[
\mcyl \times (K,E_K)\cup \1 \times (L,E_L)\to \mcyl \times (L,E_L)
\]
where $(K,E_K)\to (L,E_L)$ is a monomorphism.
\end{defn}

\begin{Remark}
In other words, a marked right anodyne extension is cellular if and only if it is in the smallest class of right anodyne extensions which contains the classes (B1) and (B2) of Remark \ref{markedgenerators}.
\end{Remark}

\begin{thm}\label{properness}
Consider a pullback square of marked simplicial sets
\[
\begin{tikzcd}
(Y,E_Y) \rar{j} \dar & (X,E_X) \dar{p}\\
(K,E_K) \rar[swap]{i} & (L,E_L)
\end{tikzcd}
\]
where $p$ is a marked left fibration and $i\colon (K,E_K)\to (L,E_L)$ is a cellular right anodyne extension. Then $j\colon (Y,E_Y) \to (X,E_X)$ is marked right anodyne.
\end{thm}

\begin{proof}
The proof is analogous to the proof for left fibrations of simplicial sets, see Proposition \ref{properleft}. The class of morphisms $(K,E_K)\to (L,E_L)$, for which the conclusion holds is saturated and satisfies the right cancellation property. Thus it suffices to show the assertion for pullback squares of the form
\[
\begin{tikzcd}
(Y,E_Y) \rar{j} \dar & (X,E_X) \dar{p}\\
\1\times (L,E_L) \rar[swap]{i} &\mcyl \times (L,E_L).
\end{tikzcd}
\]
We observe that in this case $i$ is a right deformation retract, hence by Proposition \ref{leftright} the map $j$ is a right deformation retract and thus by Lemma \ref{deformright} is a marked right anodyne extension.
\end{proof}

\begin{coro}\label{markedproperleft}
Any coCartesian fibration of simplicial sets is proper and any Cartesian fibration is smooth with respect to the elementary homotopical datum $(\Delta^1,\emptyset)$ on simplicial sets. 
\end{coro}

\begin{proof}
Translating into marked simplicial sets, we consider a pullback diagram
\[
\begin{tikzcd}
Y^\natural \rar{j}\dar[swap]{q} & X^\natural \dar{p}\\
K^\sharp \rar{i} & L^\sharp
\end{tikzcd}
\]
where $p$ and $q$ are marked left fibrations and $K^\sharp \to L^\sharp$ is cellular marked right anodyne. It follows from the previous theorem that $j$ is marked right anodyne, hence by Lemma \ref{underlyingmarked} the map of simplicial sets $Y\to X$ is right anodyne.
\end{proof}

\begin{Remark}
One can in fact show that the coCartesian fibrations are proper with respect to the elementary homotopical datum $\mathcal I^+$ on marked simplicial sets, however the proof is much more involved. By inspection of the generators for marked right anodyne extensions, the problem is to show that pulling back the maps
\[
(\Lambda^n_k)^\flat \to (\Delta^n)^\flat
\]
along a coCartesian fibration will result in a marked right anodyne extension. A proof of properness is shown in \cite[Appendix B.3]{lurieha}. 
\end{Remark}

\bibliographystyle{alpha}
\bibliography{thesis}

\end{document}